\documentclass[reqno,10pt]{amsart}  
\usepackage[english]{babel}
\usepackage{amsthm}
\usepackage{amsmath}
\usepackage{amssymb}
\usepackage{mathrsfs}
\usepackage[latin1]{inputenc}
\usepackage[normalem]{ulem}
\usepackage{graphicx}

\usepackage[all]{xy}
\usepackage{dsfont}
\usepackage[bookmarksnumbered,colorlinks]{hyperref}
\hypersetup{colorlinks=true, linkcolor=blue, citecolor=blue}
\usepackage{epsf}
\usepackage{enumerate}
\def\bb#1\eb{\textcolor{blue}
{#1}} %
\newcommand{\R}{\mathds R}

\newcommand{\be}{\begin{equation}}
\newcommand{\ee}{\end{equation}}

     \def\bb#1\eb{\textcolor{blue}{#1}} %
      
\newcommand{\RR}{\mathds{R}}
\newcommand{\vep}{\varepsilon}
\newcommand{\lra}{\longrightarrow}
\newcommand\ip[3]{g_{#3}({#1},{#2})} 
\newcommand{\ds}{\dot{\sigma}}
\newcommand{\xx}{{\boldsymbol {\it x}}} 
\newcommand\ff[2]{S_{\ds(0)}^P({#1},{#2})} 
\newcommand{\beqa}{\begin{eqnarray}}
\newcommand{\beq}{\begin{equation}}
\newcommand{\eeqa}{\end{eqnarray}}
\newcommand{\eeq}{\end{equation}}
\newcommand\chrf[1]{I^+(#1)} 
\newcommand\chrp[1]{I^-(#1)} 
\newcommand\cf[1]{J^+(#1)} 
\newcommand\cp[1]{J^-(#1)} 
\newcommand\fh[1]{E^+(#1)} 

\newcommand\uu{\mathcal{U}} 

\title[Penrose's singularity theorem in a Finsler spacetime]{Penrose's singularity theorem in a Finsler spacetime}

\author[A. B. Aazami]{Amir Babak Aazami}
\address{Kavli IPMU (WPI), UTIAS\hfill\break\indent
The University of Tokyo\hfill\break\indent
Kashiwa, Chiba 277-8583, Japan}
\email{amir.aazami@ipmu.jp}

\author[M. A. Javaloyes]{Miguel Angel Javaloyes}
\address{Departamento de Matem\'aticas \hfill\break\indent
Universidad de Murcia \hfill\break\indent
Campus de Espinardo\hfill\break\indent
30100 Espinardo, Murcia, Spain}
\email{majava@um.es}

\thanks{
\indent
This activity is supported by the programme ``Young leaders in research'' 18942/JLI/13  by Fundaci\'on S\'eneca, Regional Agency for Science and Technology from the Region of Murcia, and by the World Premier International Research Center Initiative (WPI), MEXT, Japan. 
 }

\begin{document}

\newtheorem{thm}{Theorem}[section]
\newtheorem{prop}[thm]{Proposition}
\newtheorem{lemma}[thm]{Lemma}
\newtheorem{cor}[thm]{Corollary}
\theoremstyle{definition}
\newtheorem{defi}[thm]{Definition}
\newtheorem{notation}[thm]{Notation}
\newtheorem{exe}[thm]{Example}
\newtheorem{conj}[thm]{Conjecture}
\newtheorem{prob}[thm]{Problem}
\newtheorem{rem}[thm]{Remark}
\newtheorem{conv}[thm]{Convention}
\newtheorem{crit}[thm]{Criterion}
\newtheorem{propdef}[thm]{Proposition-definition}
\newtheorem{lemmadef}[thm]{Lemma-definition}

\begin{abstract}
We translate Penrose's singularity theorem to a Finsler spacetime.  To that end, causal concepts in Lorentzian geometry are extended, including definitions and properties of focal points and trapped surfaces, with careful attention paid to the differences that arise in the Finslerian setting.
\end{abstract}

\maketitle

\section{Introduction}

The purpose of this paper is to show that the famous ``singularity theorem" of R. Penrose \cite{pen65} translates straightforwardly to the setting of a {\it Finsler spacetime}.  In the singularity theorem of Penrose (and the subsequent timelike version of S. Hawking \cite{hawking67} soon thereafter), the existence of a physical spacetime singularity is equated mathematically with the existence of an incomplete causal geodesic.  Viewed in this light, the singularity theorems are remarkable in that they provide fairly generic and purely geometric conditions under which a Lorentzian manifold fails to be geodesically complete.  Given the central role that geometry plays here, it is therefore worthwhile to examine these theorems in geometries more general than Lorentzian. 

In particular, what about the case of a \emph{Finsler spacetime,} where one has only a norm on (a subset of) the tangent bundle, whose Hessian has Lorentz signature?  Finsler spacetimes have rich geometries, characterized by their convex cones; whereas in Lorentzian spacetimes timecones are byproducts of a global metric, in a Finsler spacetime they take center stage.  Now suppose that one wishes to determine whether incomplete geodesics exist in such a setting. Given the (Lorentz) signature of the Hessian mentioned above, a natural plan of attack is to write down the (geometric) conditions required in Penrose's Lorentzian proof, and to see if they can still be defined in the Finslerian setting.  After doing this, one then proceeds to see whether the mechanism of Penrose's proof still goes through.  This, in essence, is what we are doing in this paper.  Our motivation is twofold: (1) to study geodesics in Finsler manifolds whose Hessians have Lorentz signature, and (2) to understand how ``rigid" Penrose's singularity theorem is with respect to its assumptions, in the following sense: if certain geometric quantities on the manifold\,---\,such as the metric, for example\,---\,are weakened, will geodesic incompleteness still hold as Penrose prescribed?  Indeed, since Finsler geometry has made some headway into both general relativity and high energy physics (see, e.g., \cite{Perlick06}, \cite{girelli07}, \cite{kostelecky11}, \cite{pfeifer}), it is hoped that this paper may contribute to future studies in which gravity and Finsler geometry are examined together.  

Even though the steps in Penrose's proof generalize quite well to the Finslerian setting, nevertheless there are some difficulties appearing in this process that we mention here. The first difficulty is that computations of the variations of energy are more involved, but this can be overcome via the approach taken in \cite{Ja13,Ja14,JaSo14}.  The key point in this approach is that it provides a relation between the Chern curvature tensor and the curvature of an affine connection (see \eqref{CurvatureRel}), and also between the curvature of a covariant derivative along a two-parameter map and the Chern curvature tensor (see \eqref{varCurvRel}). These relations yield the first and second variations in an index-free manner different from the classical one \cite[Chapters 5,7]{bao2000}, and allow one to recover the results of Morse index theory with less effort. As a matter of fact,  the covariant derivative along curves induced by the affine connection is the same as the one considered in the classical reference \cite{bao2000}; moreover, all the formulas in the classical Lorentzian case hold, with the addition only of some terms depending on the Cartan and Chern tensors (see \eqref{cartantensor} and \eqref{Cherntensor}), which in fact turn out to be zero in many relevant cases. Observe that even though we have used the Chern connection to define geodesics, the Ricci curvature, and trapped surfaces, none of these depends on the particular choice of this connection; all of them can be recovered using other connections, given enough compatibility with geodesic variations (see Remark~\ref{connections}).  A second difficulty that arises is that the Finslerian exponential map cannot be used to obtain a variation for arbitrarily chosen variation and acceleration fields (see Lemma \ref{AVvar}), which is required to prove Proposition \ref{prop:causal2}.  This is overcome by using a Lorentzian metric associated to a geodesic vector field, whose Levi-Civita connection can be identified with the Chern connection of the Finsler spacetime.  A third difficulty is the lack of the differentiability of the metric in timelike directions. This hypothesis is important because there is a well known family of Finsler spacetimes, sometimes referred to as static spacetimes, that are not smooth in certain timelike directions.  As most of the hypotheses of Penrose's theorem depend only on lightlike vectors, the lack of differentiability with respect to timelike vectors makes it a little harder to prove Proposition \ref{prop:causal1}.   A fourth difficulty is that one must also take care when dealing with the notion of orthogonality to submanifolds, as for example in Propositions \ref{codim2} and \ref{prop:ftrapped}; in general, however, strong convexity and transversality are enough to generalize Lorentzian results. Nonetheless, to give a purview of the difficulties that arise, observe that the fact that there are exactly two future-directed lightlike vectors orthogonal to a spacelike codimension 2 submanifold is not immediate in the Finslerian setting and requires some work to prove (see Proposition \ref{codim2}).  Finally, in this paper we do not attempt to show the equivalence of (the Finsler spacetime analogues of) global hyperbolicity and a Cauchy hypersurface.  Recall that global hyperbolicity guarantees that the causal sets $J^{\pm}(K)$ are closed for any compact $K \subset M$, which fact enables one to prove that future horizons of trapped surfaces are closed topological hypersurfaces\,---\,a key step in Penrose's proof.  In the Finsler spacetime setting, therefore, we have assumed that our Finsler spacetimes are globally hyperbolic, \emph{in addition} to having Cauchy hypersurfaces (actually, it suffices only to assume that the diamonds $\cf{p} \cap \cp{q}$ are closed for all $p,q \in M$).  Of course, if an equivalence exists between global hyperbolicity and a Cauchy hypersurface for Finsler spacetimes, then this additional assumption will be redundant.

This paper is organized as follows.  In Sections \ref{section:Finsler} and \ref{section:causality} the definition of a Finsler spacetime (following \cite{JaSa13}) is reviewed and basic causal concepts, such as timelike, spacelike, and lightlike vectors, are defined.  In Section \ref{section:curvature}  the relevant Finslerian geometric quantities\,---\,the Cartan tensor, Chern connection, and curvature\,---\,are reviewed, following \cite{Ja13,JaSo14}.  Then, after a discussion of spacelike submanifolds and Jacobi fields in Finsler spacetimes in Section \ref{section:jacobi}, we proceed in Section \ref{section:focal} with a detailed account of focal points along lightlike geodesics in Finsler spacetimes, in close analogy to the Lorentzian treatment in \cite[Chapter~10]{o1983}.  Finally, a (Finsler spacetime) version of Penrose's singularity theorem is worked out in Section \ref{section:Penrose}, in close analogy to the Lorentzian proof as presented in \cite[Chapter~14]{o1983}.

\section{Preliminaries on Finsler spacetimes}
\label{section:Finsler}
We consider here a generalization of the notion of Finsler spacetime, in the sense that a high degree of non-smoothness is allowed in timelike directions.
\begin{defi}
\label{def:Finsler}
Let $M$ be a smooth connected manifold  of dimension $n$  and $\pi\colon TM\lra M$ the natural projection from its tangent bundle $TM$.  Given an open subset $A \subset TM\setminus \bf{0}$ satisfying $\pi(A) = M$, a continuous function $L\colon A\lra (0,+\infty)$ determines a \emph{Finsler spacetime} $(M,L)$ if it satisfies   the following properties:
\begin{itemize}
\item[(i)] each $A_p:=A\cap T_pM$ is both convex (if $v, w \in A_p$ and $\lambda \in [0,1]$, then $\lambda v +(1-\lambda) w \in A_p$) and conic (if $v\in A_p$ and $\lambda>0$, then $\lambda v\in A_p$), 
\item[(ii)] $L$ is positive homogeneous of degree $2$ (if $v\in A$ and $\lambda>0$, then $L(\lambda v)=\lambda^2 L(v)$),
\item[(iii)] $A$ has smooth boundary in $TM\setminus \bf 0$, 
\item[(iv)]  $L$ extends to the  boundary  $\partial{A}$  of $A$ as zero, and this extension is smooth at
 $\bar{A}\setminus\{A\cup \bf 0\}$, in the sense that there is  a neighborhood  of the boundary where $L$ is smooth. Moreover, $\bar{A}$ is \emph{salient;} i.e., it does not contain lines,
\item[(v)]  the
\emph{fundamental tensor} $g$, defined by 
\begin{equation}\label{fundten}
g_v(u,w)\ :=\ \frac{1}{2}\!\left. \frac{\partial^2}{\partial s\partial t}L(v+tu+sw)\right|_{t=s=0}
\end{equation}
for every $v\in \hat{A} :=\bar{A}\setminus \bf 0$ at which $L$ is smooth, and for all $u,w\in T_{\pi(v)}M$,  is required to be nondegenerate on the boundary of $\hat{A}$.
\item[(vi)] the subset $\{v\in A_p: L(v)\geq 1\}$ is convex for every $p\in M$.
\end{itemize}
\end{defi}
 To avoid problems with differentiability, we will assume in fact that the function $L$ can be extended to a conic open subset $A^* \subset TM \setminus {\bf 0}$ which contains $\hat{A}$ in such a way that $L$ preserves properties $\rm (ii)$ and  $\rm (v)$  above in an open subset that contains the boundary.  In what follows, we will use the notation $\hat{A}_p=\hat{A}\cap T_pM$ for every $p\in M$.  Observe that this definition is more general than the ones considered in \cite{Beem70,Perlick06,JaSa13} and that it includes the examples in \cite{CaSt15,LaPeHa12}.  On the one hand, we consider fundamental tensors defined only with respect to causal vectors, as in \cite{JaSa13}, because causality and geodesics do not depend on the extension, which, when it exists, is not unique.  On the other hand, allowing for a lack of differentiability in timelike directions is important because it allows us to consider examples such as static Finsler spacetimes $(M\times \R,L)$, where the time-independent metric $L$ is defined as
$L(v,\tau) =  -F(v)^2+\Lambda \tau^2$,
where $(v,\tau)\in TM\times\R$ and $\Lambda$ and $F$ are, respectively, a positive function and a Finsler metric on $M$. Observe that because $F^2$ is not smooth in the zero section (see \cite[Proposition 4.1]{Warner}), $L$ is not smooth in the vectors $(0,\tau)$.  In general, the possible lack of smoothness in our definition along timelike directions prohibits the use of convex neighborhoods, even when the metric is extendable to the whole tangent bundle, as in \cite{Minguzopoli2}. Of course, when the metric $L$ is smooth and  its fundamental tensor is nondegenerate, our definition is equivalent to that considered in \cite{JaSa13} and also, using \cite{Minguzopoli2}, to that in \cite{Beem70}.  As a final remark, observe that there are some cases in which the metric is not smooth in some directions but it is still possible to define timelike geodesics (see \cite{LaPeHa12}).

Moving on, we define the {\it indicatrix} of $L$ at $p\in M$ to be the set $\Sigma_p=\{v\in A_p: L(v)=1\}$ and the {\it lightcone} ${\mathcal C}_p$ at $p$ as the boundary of $A_p$ minus the zero vector, namely, ${\mathcal C}_p=\{v\in \hat{A}_p:L(v)=0\}$.  Recall that a hypersurface $S$ is {\em strongly convex} (resp. {\em convex}) when its second fundamental form with respect to the opposite of the position vector is definite (resp. semi-definite) with respect to any transverse vector (here we compute the second fundamental form by considering $T_pM$ as an affine space endowed with the  connection associated with the affine structure); it is {\em strictly convex} when the tangent space at any point $v_0$ touches $S$ only at that point (see, e.g., \cite[Theorem 2.14]{JaSa11}).  Moreover, it is well known that strong convexity implies strict convexity, but not the converse, and that the smooth boundary of a (strictly) convex body is (strictly) convex. Therefore, Definition \ref{def:Finsler}$(vi)$ implies that $\Sigma_p$, when smooth, is convex.
\begin{prop}
\label{prop:ind}
Let $(M,L)$ be a Finsler spacetime and $p\in M$. Then   ${\mathcal C}_p\setminus 0$  is a smooth hypersurface of $T_pM$ and the following are true:
\begin{enumerate}[(i)]
\item The indicatrix $\Sigma_p$ is strongly convex, $g_v$ has index $n-1$ at the points $v\in \Sigma_p$ where $L$ is smooth and $g_v$ is non-degenerate, 
\item The tangent space to the lightcone ${\mathcal C}_p$ at $v\in {\mathcal C}_p\setminus 0$ is given by 
$\{w\in T_pM: g_v(v,w)=0\}$. Moreover,  $g_v(v,v)=L(v)=0$,
\item  The fundamental tensor $g$ is negative semi-definite when restricted to  the tangent space to ${\mathcal C}_p\setminus 0$   and for every $v\in {\mathcal C}_p\setminus 0$, $g_v$ has index $n-1$, 
\item There exists an affine hyperplane $W \subset T_pM$ that does not contain any direction of $\hat{A}_p$ and the intersection of $W$ with ${\mathcal C}_p$ is a compact strongly convex hypersurface in $W$  homeomorphic to a sphere; furthermore, $W \cap \mathcal{C}_p$ is also strictly convex,
\item ${\mathcal C}_p$ is convex and its intersection with its tangent hyperplane at $v\in {\mathcal C}_p$ is equal  to the ray from the origin that goes through $v$.
\end{enumerate}
\end{prop}
\begin{proof}
For part $(i)$, observe that at the points  $v\in \Sigma_p$ where $\Sigma_p$ is smooth,  the restriction of the fundamental tensor to the tangent spaces of $\Sigma_p$ can be expressed in terms of its second fundamental form $\sigma^\xi$ with respect to the opposite $\xi$ to the position vector, as
\begin{equation}\label{secondfun}
g(X,X)\ =\ -\frac 12 \sigma^\xi(X,X) \xi(L),
\end{equation}
 for $X$ tangent to $\Sigma_p$  (see, e.g., the relation $(2.5)$ in \cite{JaSa11}).  Recall also the well known facts \begin{equation}\label{eq:dL}
 dL_v(w)=2g_v(v,w),\quad dL_v(v)=2L(v)=2g_v(v,v).
\end{equation} 
  It follows from part $(vi)$  in Definition \ref{def:Finsler}  that $\Sigma_p$ is convex and  $\sigma^\xi$ is negative semi-definite. As \eqref{eq:dL} implies that the direction $v$ is $g_v$-orthogonal to the tangent space to $\Sigma_p$, the nondegeneracy of $g_v$ and \eqref{secondfun} imply that $g_v$ has index $n-1$ and $\Sigma_p$ is strongly convex. 
 
 Part $(ii)$ follows from \eqref{eq:dL}, which also implies that  ${\mathcal C}_p\setminus 0$ is a smooth embedded hypersurface of $T_pM$. Part $(iii)$ is a consequence of parts $(i)$ and $(ii)$: by part $(i)$, we deduce that the fundamental tensor has index $n-1$, since it has that index in points of $A$ where $L$ is non-degenerate, which can be chosen as close as we want to $\mathcal C_p$. Then part $(ii)$ implies that $g_v$ is negative semi-definite when restricted to the tangent spaces of $\mathcal C_p$ (it is the orthogonal space of a lightlike vector in the semi-Riemannian space $(T_pM,g_v)$ of index $n-1$). Finally, for $(iv)$ and $(v)$ first observe that \eqref{eq:dL} still remains true at the points of $\mathcal C_p$, choosing any $\xi$ such that $-\xi$ points to $A$. Then  $-\xi(L)>0$, because $L$ is zero on the lightcone and positive on $A$, but $-\xi(L)\not= 0$ since otherwise $g$ would not have index $n-1$. Moreover, the only degenerate direction in $T_v{\mathcal C}_p$ is the one determined by $v$, which is in fact contained in ${\mathcal C}_p$, and $\sigma^\xi$ is negative definite in the tangent to $W\cap \mathcal C_p$.  For the choice of $W$, take into account that $\bar{A}$ is convex and salient.  In particular, for $(iv)$ take a tangent plane to the cone ${\mathcal C}_p$, spin it around the origin in the opposite direction of the cone and then consider any parallel affine hypersurface to it contained in the same half-space as the cone ${\mathcal C}_p$. 
\end{proof}

\begin{rem}\label{rem:Lnegative}
Observe that part $(ii)$ of Proposition \ref{prop:ind} implies that if we extend $L$ to an open subset $A^*$, then we can assume, by shrinking $A^*$ if necessary, that $L$ is negative in $A^*\setminus \hat{A}$. This is because given $p\in M$, $dL_v(w)=2g_v(v,w)$ for every $v\in \mathcal C_p$.  By part $(ii)$ of Proposition \ref{prop:ind}, $g_v(v,w)=0$ if and only if $w$ is tangent to $\mathcal C_p$, hence we conclude easily that $g_v(v,w)>0$ if and only if $w$ points to the convex subset delimited by $\mathcal C_p$, which together with the homogeneity of $L$ and the compactness of the intersection ${\mathcal C}_p\cap W$ for a hyperplane $W$ as in part $(iv)$ of Proposition \ref{prop:ind}, allows us to conclude the existence of the open subset $A^*$ as above.
\end{rem}
 Next, we have the following analogues of the reverse triangle inequality and the reverse fundamental inequality; see also \cite{Ming14} for a proof of this result in a smooth background.  Just as the fundamental inequality in Riemann-Finsler geometry (see \cite[p.~10]{bao2000}) generalizes the Cauchy-Schwarz inequality, so the reverse fundamental inequality in Proposition~\ref{Lorentznorm} below generalizes the timelike Cauchy-Schwarz inequality in Lorentzian geometry (see \cite[Proposition 30,~p.~144]{o1983}).

\begin{prop}\label{Lorentznorm}
Consider $F=\sqrt{L}$, which is positive homogeneous of degree $1$. Then
\begin{enumerate}[(i)]
\item  $F$ satisfies the {\rm reverse triangle inequality:}
\[F(v+w)\ \geq\ F(v)\ +\ F(w)\]
at every $p\in M$ and $v,w\in \hat{A}_p$,
\item $F$ satisfies the {\rm reverse fundamental inequality:}
\[g_v(v,w)\ \geq\ F(v) F(w)\]
for all $v\in \hat{A}$ at which $L$ is smooth and $g_v$ is nondegenerate, and for all $w\in \hat{A}_{\pi(v)}$, with equality if and only if $w=\lambda v$ for some $\lambda> 0$.
\end{enumerate}
\end{prop}
\begin{proof}
Part $(i)$ is a straighforward consequence of condition $(vi)$ in Definition \ref{def:Finsler}.
For part $(ii)$,  observe that if $v\in A$, then every $w\in \hat{A}_{\pi(v)}$ can be expressed as $w=\lambda v+u$, where $u$ is $g_v$-orthogonal to $v$.  Furthermore, considering the plane $\pi_{v,w}={\rm span}\{v,w\}$ and the fact that $\Sigma_p\cap \pi_{v,w}$ is a convex curve in $T_{\pi(v)}M$ converging to the boundary of $A$, we easily deduce that $u\notin \hat{A}_{\pi(v)}$; this implies that $\lambda>0$, because otherwise $w$ cannot belong to $\hat{A}_{\pi(v)}$.  Then
$g_v(v,w)=\lambda g_v(v,v)=\lambda F(v)^2$. Now,  by part $(i)$ of  Proposition \ref{prop:ind},  $\Sigma_p$ is strongly convex at  $v/F(v)$  and it remains on one side of its tangent plane at $v/F(v)\in\Sigma_p $, touching  the tangent hyperplane only at $v/F(v)$.  As $w/F(v)=\lambda v/F(v)+u/F(v)$, it follows that $\lambda\geq F(w)/F(v)$  and the equality occurs only when $v$ and $w$ are linearly dependent.   Recalling that $g_v(v,w)=\lambda F(v)^2$,  this concludes the reverse fundamental inequality when $v\in A$.  If $v$ belongs to the boundary of $A$, then the reverse fundamental inequality is an easy consequence of Remark \ref{rem:Lnegative}.
\end{proof}

\section{Causality}
\label{section:causality}

\begin{defi}
\label{defi1}
Let $(M,L)$ be a Finsler spacetime.  Tangent vectors to $M$ are classified as follows:
$$
v \in T_pM~~\text{is}~~\left\{\begin{array}{rcl}
{\it timelike} & \text{if} & v~\text{or}~-\!v~\text{belongs to}~A_p,\\
{\it lightlike} & \text{if} & v~\text{or}~-\!v~\text{belongs to}~\hat{A}_p\setminus A_p,\\
{\it causal} & \text{if} & v~\text{or}~-\!v~\text{belongs to}~\hat{A}_p,\\
{\it spacelike} & \text{if} & \text{it is not causal.}
\end{array}\right.
$$
Moreover, a causal vector $v \in T_pM$ is {\it future-pointing} if  $v \in \hat{A}_p$ and  {\it past-pointing} if  $ -v \in \hat{A}_p$.
\end{defi}

 If the smooth manifold $M$ is simultaneously endowed with a Finsler spacetime structure $(M,L)$ and a time-oriented Lorentzian metric $(M,g)$, then we will use the designations ``$g$-timelike" and ``$g$-causal" to refer to the causality of vectors in $(M,g)$, reserving the designations ``timelike" and ``causal" for $(M,L)$.  Next, we say that a curve is future-pointing (resp. past-pointing) timelike, lightlike, causal, or spacelike, when its tangent vector is future-pointing (resp. past-pointing) timelike, lightlike, causal, or spacelike.  By a ``future-pointing  piecewise smooth causal curve $\alpha\colon [a,b] \lra M$," we always refer to a causal curve satisfying $\dot{\alpha}(t_i^-), \dot{\alpha}(t_i^+) \in \hat{A}_{\alpha(t_i)}$ at any break $\alpha(t_i)$; in other words, piecewise smooth causal curves have tangent vectors that always stay inside $\hat{A}$, even at breaks.  Moreover, we say that two points $p,q\in M$ are \emph{chronologically related,} denoted $p\ll q$, when there exists a piecewise smooth future-pointing timelike curve from $p$ to $q$, and \emph{causally related,} denoted $p\leq q$, when there exists a piecewise smooth future-pointing causal curve from $p$ to $q$.  We then have the usual designations:
$$
\left\{\begin{array}{rccl}
\text{\emph{chronological future}}: & \chrf{p} \!\!&=&\!\! \{q\in M\,:\,p\ll q\},\\
\text{\emph{chronological past}}: & \chrp{p} \!\!&=&\!\! \{q\in M\,:\,q\ll p\},\\
\text{\emph{causal future}}: & \cf{p} \!\!&=&\!\! \{q\in M\,:\,p\leq q\},\\
\text{\emph{causal past}}: & \cp{p} \!\!&=&\!\! \{q\in M\,:\,q\leq p\}. \nonumber
\end{array}\right.
$$
As in Lorentzian spacetimes, chronological sets are always open sets (see \cite[Proposition 3.7]{JaSa11}).  Denoting by $I^{\pm}(B,\uu)$ the chronological future/past of $B$ in any open set $\uu$, it follows that the sets $I^{\pm}(B,\uu)$ are open and that $I^{\pm}(B,\uu) \subset I^{\pm}(B) \cap \uu$.  Finally, we define the notions of \emph{achronality, edge points,} and \emph{past} and \emph{future sets} exactly as in the Lorentzian setting (see \cite[p.~413-5]{o1983}).  We postpone to Section~\ref{section:Penrose} a discussion of Cauchy hypersurfaces in Finsler spacetimes.  Finally, let us see that we can always select a timelike vector field globally defined on our spacetime.
\begin{prop}
\label{lemma:Lorentz}
Let $(M,L)$ be a Finsler spacetime.  Then there exists a smooth vector field $\tau \in \mathfrak{X}(M)$ satisfying $\tau_p \in A_p$ for all $p \in M$.
\end{prop}
\begin{proof}

Observe that at each point $p\in M$ we can find easily a neighborhood $V_p$ of $p$ that admits a vector field $T^p$ satisfying $T^p(q)\in A$ for all $q\in V_p$. Consider now the covering $\{V_p\}_{p\in M}$ of $M$. As $M$ is paracompact, we can extract a locally finite covering $\{V_i\}_{i\in I}$, by virtue of which the vector field $\tau =\sum_{i\in I} \mu_i T^i$ will be well-defined, where $\mu_i$ is a partition of the unity associated to the covering.  Furthermore, $\tau_p \in A_p$ for every $p\in M$ because of the convexity of $A_p$.
\end{proof}

\section{Cartan tensor, Chern connection, and curvature}
\label{section:curvature}
 The contents of this section apply more generally to any (conic) \emph{pseudo-Finsler metric,} that is, a pair $(M,L)$ as in Definition \ref{def:Finsler}, but with the additional requirements that $L$ be smooth on $A$ and that the fundamental tensors be nondegenerate (though they are not required to have Lorentzian index and $L$ can be zero or negative).  Consult \cite{Ja13,Ja14,JaSo14} for a more thorough discussion of much of the contents of this section.

Let us begin by defining the \emph{Cartan tensor} associated to any pseudo-Finsler metric $L\colon A\subset TM\lra \R$ to be
\begin{equation}\label{cartantensor}
C_v(w_1,w_2,w_3)\ =\ \frac 14\left. \frac{\partial^3}{\partial s_3\partial s_2\partial s_1}L\left(v + \sum_{i=1}^3 s_iw_i\right)\right|_{s_1=s_2=s_3=0},
\end{equation}
for $v\in A$ and $w_1,w_2,w_3\in T_{\pi(v)} M$.
Observe that $C_v$ is symmetric (its value does not depend on the order of $w_1$, $w_2$, and $w_3$), positive homogeneous of degree $-1$ ($C_{\lambda v}=\frac{1}{\lambda} C_v$ for every $v\in A$ and $\lambda>0$), and 
\begin{equation}\label{propCartan}
C_v(v,w_1,w_2)\ =\ C_v(w_1,v,w_2)\ =\ C_v(w_1,w_2,v)\ =\ 0
\end{equation}
(see, e.g., \cite[subsection 2.2]{JaSo14}).  Now fix an open subset $\Omega\subset M$, denote by $\mathfrak{X}(\Omega)$ the space of smooth vector fields on $\Omega$, and define a vector field $V$ on $\Omega$ to be \emph{$L$-admissible} if $V(p) \in A_p$ for every $p\in \Omega$. The Chern connection can be interpreted as an affine connection $\nabla^V$ which is torsion-free and almost compatible with $g$ (see, e.g., \cite[Section 2.1]{Ja13} and references therein).  Moreover, given a smooth curve $\gamma$ and an $L$-admissible vector field $W$ along $\gamma$, the Chern connection determines a covariant derivative along $\gamma$ with reference vector $W$, denoted by $D_{\gamma}^W$, that is almost $g$-compatible, namely,
\begin{equation}\label{almostmetric}
\frac{d}{dt}\left(g_{W}(X,Y)\right)\ =\ g_{W}(D_\gamma^WX,Y)\ +\ g_{W}(X,D_\gamma^WY)\ +\ 2C_{W}(D_\gamma^WW,X,Y),
\end{equation}
for any $X,Y\in\mathfrak{X}(\gamma)$.
We then define a {\it geodesic} of $(M,L)$ to be a smooth $L$-admissible curve $\gamma$ such that $D_\gamma^{\dot\gamma}\dot\gamma=0$.  We say that a vector field $X$ along $\gamma$ is  {\it $\gamma$-parallel} if $D^{\dot\gamma}_{\gamma}X=0$. Observe that if $\gamma$ is a geodesic and $X$ is $\gamma$-parallel along $\gamma$, then $g_{\dot\gamma}(\dot\gamma,X)$ is constant along $\gamma$ by \eqref{almostmetric}, and if $X$ and $Y$ are $\gamma$-parallel, then $g_{\dot\gamma}(X,Y)$ is also constant along $\gamma$. In particular, given a $g_{\dot\gamma}$-orthonormal system in $T_{\gamma(a)}M$, $\gamma$-parallel transport gives an orthonormal system along $\gamma$.  

Using the Chern connection, we can define for every $L$-admissible vector field $V \in {\mathfrak{X}(\Omega)}$ the curvature tensor associated to $V$ as  
\[R^V(X,Y)Z\ =\ \nabla^V_X\nabla^V_YZ-\nabla^V_Y\nabla^V_XZ-\nabla^V_{[X,Y]}Z,\]
for every $X,Y,Z\in \mathfrak{X}(\Omega)$.  Moreover, if we define
\begin{equation}\label{Cherntensor}
P_V(X,Y,Z)\ =\ \left.\frac{\partial}{\partial t}\left( \nabla_X^{V+tZ}Y\right)\right.\bigg|_{t=0},
\end{equation}
then 
\begin{equation}\label{CurvatureRel}
R_V(X,Y)Z\ =\ R^V(X,Y)Z\ -\ P_V(Y,Z,\nabla^V_XV)\ +\ P_V(X,Z,\nabla^V_YV),
\end{equation}
where $R_V(X,Y)Z$ is the Chern curvature tensor defined in \cite[Formula (3.3.2) and Exercise 3.9.6]{bao2000}. The tensor $R_V$ does not depend on the extension of $V$ used to compute it,  and it is used to compute the flag curvature of a pseudo-Finsler metric (see \cite{Ja14}).  In fact, for every $v\in A$ with $p=\pi(v)$, we can define a trilinear map $R_v:T_pM\times T_pM\times T_pM\lra T_pM$ in such a way that, if $V,X,Y,Z$ are extensions to $\Omega$ of $v,x,y,z\in T_pM$, then
$R_v(x,y)z=(R_V(X,Y)Z)(p).$  In order to compute the flag curvature, we have to fix a flagpole $v\in A$ and a vector $w\in T_{\pi(v)}M$ with $L(v)L(w)-g_v(v,w)^2\not=0$.  Then
\[K_v(w)\ =\ \frac{g_v(R_v(v,w)w,v)}{L(v)L(w)-g_v(v,w)^2}\cdot\]
We can also define the flag curvature along curves, as follows.
Given a piecewise smooth variation $\xx\colon [0,b] \times (-\delta,\delta) \lra M$, $(u,v)\mapsto\xx(u,v)$, we henceforth adopt the following notation: $\sigma_v=\xx(\cdot,v)$ and $\beta_u=\xx(u,\cdot)$ for all $u\in[0,b]$ and $v\in(-\delta,\delta)$. Then we define
\beqa
\label{eqn:Rsigma1}
R^\xx(\tilde{Y})\ :=\ D_{\sigma_v}^{\dot\sigma_u}D_{\beta_u}^{\dot\sigma_u}\tilde{Y}-D_{\beta_u}^{\dot\sigma_u}D_{\sigma_v}^{\dot\sigma_u}\tilde{Y}
\eeqa
for any vector field $\tilde{Y}$ along the variation $\xx$ (see \cite[Remark 1.2]{Ja14}). As this quantity depends only on the curve $\sigma$ and the variation field $Z$ of $\xx$, we will write
\beqa
\label{eqn:Rsigma2}
R^\sigma(\dot\sigma,Z)Y\ :=\ R^\xx(\tilde{Y}),
\eeqa
where $Y$ is a vector field along $\sigma$ and $\tilde{Y}$ an extension of $Y$ to the variation $\xx$. In general, we have 
\begin{equation}\label{varCurvRel}
R^\sigma(\dot\sigma,Z)Y\ =\ R_{\dot\sigma}(\dot\sigma,Z)Y\ +\ P_{\dot\sigma}(Z,Y,D_\sigma^{\dot\sigma}\dot\sigma)\ -\ P_{\dot\sigma}(\dot\sigma,Y,D_\sigma^{\dot\sigma}Z)
\end{equation}
(see \cite[Theorem  1.1]{Ja14}), but when $\sigma$ is a geodesic it is the case that
\[R^\sigma(\dot\sigma,Z)\dot\sigma\ =\ R_{\dot\sigma}(\dot\sigma,Z)\dot\sigma\]
(see \cite[Corollary 1.3]{Ja14}) and 
\begin{equation*}
K_v(w)\ =\ \frac{g_v(R^\sigma(\dot\sigma,W)W(t_0),v)}{L(v)g_{v}(w,w)-g_v(v,w)^2},
\end{equation*}
where $\sigma$ is the geodesic passing through $\pi(v)$ with $\dot\sigma(t_0)=v$ and $W$ is any vector field along $\sigma$ satisfying $W(t_0)=w$ (see \cite[Remark 2.3]{Ja14}). Finally, we define the {\it scalar Ricci curvature} as the trace with respect to $g_v$ of the linear operator
\[{\mathbf R}_v\colon T_{\pi(v)}M\times T_{\pi(v)}M\lra\R\hspace{.03in},\hspace{,2in}{\mathbf R}_v(u,w)\ =\ g_v(R_v(v,u)w,v),\]
for any $u,w\in T_{\pi(v)}M$. The scalar Ricci curvature is a positive homogeneous function ${\rm Ric}\colon A\lra\R$  of degree zero.
\begin{lemma}\label{lem:ric}
Let $(M,L)$ be a Finsler spacetime. If $z$ is lightlike and $e_3,\ldots,e_n$ is a system of $(-g_z)$-spacelike, $g_z$-orthonormal vectors that are $g_z$-orthogonal to $z$, then ${\rm Ric}(z)= -\sum_{i=3}^n {\mathbf R} _z(e_i,e_i)$.
\end{lemma}
\begin{proof}
Let $ \tilde{z}\in T_{\pi(z)}M$ be a vector in the $g_z$-orthogonal vector space to ${\rm span}\{e_3,\ldots,e_n\}$ satisfying $g_z(\tilde{z},\tilde{z})=0$ and $g_z(z,\tilde{z}) = -1/2$.  Then  the vectors $\tilde{e}_1 := \tilde{z}+z$ and $\tilde{e}_2 := \tilde{z} - z$ are $-g_{z}$-unit timelike and spacelike, respectively.  So with respect to the $g_{z}$-orthonormal frame $\tilde{e}_1,\tilde{e}_2,e_3,\dots,e_n$, taking the trace yields
\beqa
\text{Ric}(z) &=& \underbrace{\,g_z(R_z(z,\tilde{e}_1)\tilde{e}_1,z)\, -\, g_z(R_z(z,\tilde{e}_2)\tilde{e}_2,z)}_{0}\ -\ \sum_{i=3}^n\,g_z(R_z(z,e_i)e_i,z)\nonumber\\
&=& -\sum_{i=3}^n\,g_z(R_z(z,e_i)e_i,z)\ =\ -\sum_{i=3}^n {\mathbf R}_z(e_i,e_i),\nonumber
\eeqa
 because
$$\ip{R_z(z,\tilde{z}+z)(\tilde{z}+z)}{z}{z}\ =\ \ip{R_z(z,\tilde{z}-z)(\tilde{z}-z)}{z}{z}.
$$
For the last relation, observe that $R_z$ is anti-symmetric in the first two components and
$g_z(R_z(z,\tilde{z})z,z)=0$ (this can be checked using the symmetry of $R^V$ in \cite[Proposition 3.1]{Ja13} with $V$ a geodesic field extending $z$, along with \cite[Lemma 1.2]{Ja14} and \cite[Lemma 3.10]{JaSo14}).
\end{proof}

\section{Spacelike submanifolds and Jacobi fields}
\label{section:jacobi}
 Given $p\in M$, we will say that a subspace $Q \subset T_pM$ is {\it spacelike} if it consists entirely of 
spacelike vectors. A submanifold $P\subset M$ is {\it spacelike} if $T_pP\subset T_pM$ is spacelike for every $p\in P$.  Given $p \in P$, we say that a vector $z\in A_p$ is \emph{orthogonal to $P$} if $g_z(z,v)=0$ for every $v\in T_pP$. 
Suppose now that $z$ is orthogonal to $P$ and $g_z|_{T_pP\times T_pP}$ is nondegenerate.  Then we have the splitting 
 \begin{equation}\label{split}
 T_pM\ =\ T_pP\ \oplus\ (T_pP)_z^\perp,
 \end{equation}
where $(T_pP)_z^\perp$ is the subspace of vectors orthogonal to $T_pP$. We define the {\it second fundamental form of $P$ in the direction $z$}, denoted by $S_z^P\colon T_pP\times T_pP\lra (T_pP)_z^\perp$, as $S_z^P(u,w)= {\rm nor}^P_z (\nabla^Z_UW)_p$, where ${\rm nor}^{P}_z$ is the projection to $(T_pP)_z^\perp$ via the splitting \eqref{split} and $Z$, $U$, and $W$  are arbitrary extensions of $z$, $u$ and $w$ in such a way that $U$ and $W$ are tangent to $P$ along $P$.  Moreover, we define the {\it normal second fundamental form of $P$ in the direction $z$}, denoted by $\tilde{S}_z^P\colon T_pP\lra T_pP$, as $\tilde{S}_z^P(u)={\rm tan}^P_z (\nabla^Z_UZ)_p$, with ${\rm tan}^P_z$ being the projection onto $T_pP$ by \eqref{split}. Observe that $S_z^P$ and $\tilde{S}_z^P$ are well-defined. Furthermore,  $S_z^P$ is bilinear and symmetric, $\tilde{S}_z^P$ is linear, and 
\begin{equation}\label{relationS}
g_z(S^P_z(u,w),z)\ =\ -g_z(\tilde{S}_z^P(u),w).
\end{equation}
 When we fix a smooth vector field $N$  along $P$ which is orthogonal to $P$ at every point, then $S_N^P$ and $\tilde{S}_N^P$ determine tensors (see \cite[subsection 3.1]{Ja13}). We will define {\it the mean curvature vector field of $P$ in the direction $z$} (with $z$ being orthogonal to $P$ at $p$), denoted by $H^P_z$, as the trace of $S_z^P$ with respect to $g_z$.  Finally, we will see that the second fundamental form is always defined for spacelike submanifolds.
\begin{lemma}
\label{lemma:spacesub}
 Let $P$ be a spacelike submanifold in a Finsler spacetime $(M,L)$ with $p \in P$ and $z \in \hat{A}_p$ a causal vector at which $L$ is smooth, $g_z$ is nondegenerate, and that is orthogonal to P at $p$.  Then $g_z|_{T_pP \times T_pP}$ is nondegenerate.
\end{lemma}

\begin{proof}
If $z$ is timelike, then because $g_z$ has index $n-1$ and $g_z(z,z)>0$, it follows that $g_z$ is negative-definite on the  subspace that is $g_z$-orthogonal to $z$. In particular, $P$ is contained in this orthogonal subspace and is nondegenerate.  If $z$ is lightlike, then the $g_z$-orthogonal subspace to $z$ is degenerate, since $g_z(z,z)=0$, but the degenerate line is precisely the one in the direction of $z$. As a spacelike submanifold cannot contain this line, it has to be nondegenerate.
\end{proof}
\begin{prop}\label{codim2}
Let $(M,L)$ be a Finsler spacetime and $P$ a spacelike submanifold of $M$ of codimension $2$. Then at every point $p\in P$ there are exactly two future-pointing lightlike directions orthogonal to $P$.
\end{prop}
\begin{proof}
Consider a spacelike hyperplane $W\subset T_pM$ such that $T_pP\subset W$.  Then by part $(iv)$ of Proposition \ref{prop:ind},  the intersection  ${\mathcal C}_p\cap W$  is a compact strongly convex subset homeomorphic to a sphere and the tangent planes to it touch only at one point. Therefore $T_pP$ is tangent to ${\mathcal C}_p \cap W$ at exactly two points $v_1,v_2$, because ${\mathcal C}_p \cap W$ remains in one of the half-spaces determined by $T_pP$ (if we consider an affine hyperplane $Q$ parallel  to $T_pP$ which does not intersect ${\mathcal C}_p \cap W$, then $v_1$ and $v_2$ are  the points that maximize and minimize the distance with respect to $Q$).  In particular, $T_pP$ is tangent to ${\mathcal C}_p$ at  $v_1,v_2$ and, by homogeneity, at the rays from the origin passing through them. By part $(ii)$ of Proposition  \ref{prop:ind}, this means that $v_1,v_2$ give all the lightlike directions orthogonal to $P$.
\end{proof}

Consider a geodesic $\sigma\colon[0,b]\lra M$ in a Finsler spacetime $(M,L)$ as in Section \ref{section:curvature}, namely, assume that $L$ is smooth on $\dot\sigma(s)$ for $s\in [0,b]$.  We then say that a vector field along $\sigma$ is a {\it Jacobi field} if it satisfies the equation
\[ J''\ =\ R^\sigma(\dot\sigma,J)\dot\sigma.\]
Henceforth we use the notation $J'=D_{\sigma}^{\ds} J$.
Next, given a submanifold $P$ of $M$ such that $\sigma$ is $g_{\dot\sigma(0)}$-orthogonal to $P$ at $\sigma(0)$, we say that the Jacobi field $J$ is {\it $P$-Jacobi} if $J(0)$ is tangent to $P$ and ${\rm tan}^P_{\dot\sigma}J'(0)=\tilde{S}^P_{\dot\sigma}(J(0))$.  For Section \ref{section:focal} below, we note here that the vector space of all $P$-Jacobi fields along $\sigma$ that are $g_{\ds}$-orthogonal to $\sigma$ has dimension $n-1$ (see \cite[Lemma 3.14]{JaSo14} and \cite[p. 283]{o1983}).  Moreover, we say that an instant $t_0\in (0,b]$ is {\it $P$-focal} if there exists a nonzero $P$-Jacobi field such that $J(t_0)=0$.

\begin{lemma}\label{P-Jacobi}
If $J_1$ and $J_2$ are $P$-Jacobi fields along a geodesic $\sigma\colon [0,b] \lra M$, then $g_{\dot\sigma}(J_1,J'_2) = g_{\dot\sigma}(J'_1,J_2)$.
\end{lemma}
\begin{proof}
Taking into account \cite[Proposition~3.18]{JaSo14}, the proof is the same as in the Lorentzian case.
\end{proof}

Because in the following lemma we speak about the first focal point along a geodesic, and because our definition of Finsler spacetime does not ensure the existence of convex neighborhoods, we observe here that, using an idea similar to \cite[Remark 3.2]{JaPi06a}, we can show that there are no focal points arbitrarily close to the initial point and hence, as focal points constitute a closed subset of $(0,b]$, there necessarily exists a first focal point.  Indeed, in the notation of \cite[Remark 3.2]{JaPi06a}, and assuming that the first $r$ coordinates in the system generate the tangent space to $P$, with $S:\R^{r}\rightarrow \R^r$ the coordinate counterpart of $\tilde{S}^P_{\dot\sigma}$, an instant $t$ is $P$-focal if and only if the linear map $A(t):\R^r\times \R^{n-1} \lra \R^r\times \R^{n-1}$, defined as
\[ A(t)(\tilde{v},\tilde{w})\ =\ \Phi_{11}(t)  
\left(\!\!\begin{array}{c}
\tilde{v}\\
0
\end{array}\!\!\right)
\ +\ \Phi_{12}(t)
\left(\!\!\begin{array}{c}
S(\tilde{v})\\
\tilde{w}
\end{array}\!\!\right),
\]
is singular.  Because $A(a) (\tilde{v},\tilde{w})=(\tilde{v},0)$ and $A'(a)(\tilde{v},\tilde{w})=(S(\tilde{v}),w)$, we easily conclude that it is nonsingular for all $t\in (a,a+\vep)$, for $\vep > 0$ small enough.

\begin{lemma}\label{Jacobi-P2}
Let $P$ be a spacelike submanifold of a Finsler spacetime $(M,L)$ and $\sigma\colon[0,b]\lra M$
a geodesic that is orthogonal to $P$ at $\sigma(0) \in P$.  Assume that $r > 0$ is the first focal point of $P$ along $\sigma$ and let $J$ be a nonzero $P$-Jacobi field along $\sigma$ that satisfies $J(r) = 0$.  Suppose in addition that $J(0) \in T_{\sigma(0)}P$ is nonzero. Then the following are true:
\begin{enumerate}[(i)]
\item $J$ can never be tangent to $\sigma$ on $(0,r)$ and is in fact $g_{\ds}$-orthogonal to $\sigma$,
 \item $J'(r)$ is not tangent to $\sigma$ at $r$,
 \item $J'$ is $g_{\ds}$-orthogonal to $\sigma$.
\end{enumerate}
\end{lemma}
\begin{proof}
Taking into account \cite[Lemma 3.17]{JaSo14}, the proof follows the same lines as in the Lorentzian case.
\end{proof}

\section{Focal points in Finsler spacetimes}
\label{section:focal}
 In this section we lay out in detail the necessary results regarding focal points along lightlike geodesics in Finsler spacetimes.  All results in this section have well-known Lorentzian analogues, and our treatment parallels the Lorentzian treatment as it is derived in \cite[Chapter~10]{o1983}.  Having said that, in the Finslerian setting some modifications are required, which we have been careful to write out explicitly.

Let $P \subset M$ be a spacelike submanifold and let $\sigma\colon[0,b] \lra M$ be a lightlike geodesic orthogonal to $P$ at $\sigma(0) \in P$, with endpoint $\sigma(b) = q$.  If $C_L(P,q)$ is the space of $L$-admissible piecewise smooth curves from $P$ to $q$, then we define $T_{\sigma}^{\perp}C_L(P,q)$ to be the vector space of all piecewise smooth vector fields $V$ along $\sigma$ satisfying $V(0) \in T_{\sigma(0)}P, V(b) = 0$, and such that $V(u)$ is $g_{\ds(u)}$-orthogonal to $\ds(u)$ at each $ u \in [0,b]$.  Next, given a smooth  $(P,q)$ variation  $\xx\colon [0,b] \times (-\delta,\delta) \lra M$ of $\sigma$ with variation field $V$, the second variation of energy with respect to $\xx$ is
$$
E_{\xx}''(0)\ =\ \int_0^b \Big[\ip{V'}{V'}{\ds} - \ip{R^{\sigma}(\ds,V)V}{\ds}{\ds}\Big]du\, -\, \ip{\ds(0)}{\ff{V(0)}{V(0)}}{\ds(0)}
$$
(for a proof, see, e.g., \cite[Corollary 3.8]{JaSo14}).  Observe that $S_{\dot\sigma(0)}^P$ is well defined by Lemma \ref{lemma:spacesub}.  From now on we will consider variations that include non-causal curves, since we assume that $L$ can be extended to an open subset $A^*$ that contains $\hat{A}$. Indeed, by taking a smaller interval of variation if necessary, we can assume that an arbitrary variation is contained in $A^*$ and that $L$ is smooth along the tangent vectors to the curves in the variation.  Finally, let $I^P_\sigma$ denote the corresponding index form of $E$, defined as
$$
I^P_{\sigma}(V,W)\ :=\ \int_0^b \Big[\ip{V'}{W'}{\ds} - \ip{R^{\sigma}(\ds,V)W}{\ds}{\ds}\Big]du\, -\, \ip{\ds(0)}{\ff{V(0)}{W(0)}}{\ds(0)}.
$$
With that said, we begin with a lemma (cf. \cite[Proposition 10.41,~p.~291]{o1983}).

\begin{lemma}
\label{lemma:Hessian}
Let $P$ be a spacelike submanifold of a Finsler spacetime $(M,L)$ and $\sigma\colon [0,b] \lra M$ a future-pointing lightlike geodesic from $\sigma(0) \in P$ to $\sigma(b) = q$ that is orthogonal to $P$.  If there are no focal points of $P$ along $\sigma$, then the index form $I^P_{\sigma}$ is negative semidefinite on $T_{\sigma}^{\perp}C_L(P,q)$.  If $I^P_{\sigma}(V,V) = 0$ for any $V \in T_{\sigma}^{\perp}C_L(P,q)$, then $V$ is tangent to $\sigma$.
\end{lemma}

\begin{proof}
The proof can be carried out following the same lines as in \cite[Proposition 10.41]{o1983}.  Here we only detail the choice of the basis of $P$-Jacobi fields.  Let $\{Y_1,\dots,Y_{n-1}\}$ be a basis for the space of $P$-Jacobi fields along $\sigma$ that are $g_{\ds}$-orthogonal to $\sigma$ and such that $Y_1(0)=Y_2(0)=\cdots=Y_{n-1-r}(0)=0$, with $r=\dim P$.  Since there are no focal points along $\sigma$, it follows that for each $u \in (0,b]$, no nontrivial linear combination of the $Y_i(u)$'s can be zero, hence $\{Y_1(u),\dots,Y_{n-1}(u)\}$ is a basis for $\ds(u)^{\perp} = \{z \in T_{\sigma(u)}M : \ip{\ds(u)}{z}{\ds(u)} = 0\}$.  For any $V \in T_{\sigma}^{\perp}C_L(P,q)$, we can therefore write $V = \sum_{i=1}^{n-1} f_iY_i$ for some piecewise smooth functions $f_i$ on $(0,b]$, which functions can be extended continuously to 0, since  for $i=1,\ldots,n-r-1$, $Y_i(u)=u \tilde{Y}_i(u)$, with $\tilde{Y}_1(u),\ldots,\tilde{Y}_{n-r-1}(u),Y_{n-r}(u),\ldots,Y_{n-1}(u)$ linearly independent.  Having at hand the expression $V = \sum_{i=1}^{n-1} f_iY_i$ and taking Lemma \ref{P-Jacobi} into account, $\ip{Y_i}{{Y_j}'}{\ds} = \ip{{Y_i}'}{Y_j}{\ds}$, as well as bearing in mind that the curvature along a variation behaves in a simpler way along geodesics (see \cite[Lemma 3.10]{JaSo14}), the proof then goes though as in \cite[Proposition 10.41]{o1983}.
\end{proof}

Our next proposition is a direct analogue of \cite[Proposition~43,~p.~292]{o1983}.  Recall that $H^{P}_z$ is the mean curvature vector field of $P$ in the orthogonal direction $z$.
\begin{prop}
\label{prop:Ric}
Let $P$ be an $(n-2)$-dimensional spacelike submanifold in a Finsler spacetime $(M,L)$ of {\rm dim}\,$M \geq 3$.  Let $\sigma\colon [0,b] \lra M$ be a future-pointing lightlike geodesic that is $g_{\ds(0)}$-orthogonal to $P$ at $\sigma(0) \in P$.  If
\begin{enumerate}[(i)]
\vskip 4pt
\item $\ip{\ds(0)}{H^{P}_{\dot\sigma(0)}}{\ds(0)}\ =:\ k > 0,$
\vskip 4pt
\item ${\rm Ric} (\ds)\ \geq\ 0$,
\vskip 4pt
\end{enumerate}
then there is a focal point $\sigma(r)$ of $P$ along $\sigma$ with $0 < r \leq 1/k$, provided $\sigma$ is defined on this interval.
\end{prop}

\begin{proof}
Let $e_3,\dots,e_n$ be a $g_{\ds(0)}$-orthonormal basis for $T_{\sigma(0)}P$ and $\sigma$-parallel translate them along $\sigma$ to obtain smooth vector fields $E_3,\dots, E_n \in \mathfrak{X}(\sigma)$.  Recall that as $\sigma$ is a geodesic, $E_3,\dots, E_n$ are $g_{\dot\sigma}$-orthonormal and $g_{\dot\sigma}(\dot\sigma,E_i) =0$ for $i=3,\ldots,n$.  Then the proof goes through as in \cite[Proposition~43,~p.~292]{o1983}, after taking into account Lemma \ref{lem:ric}.
\end{proof}

In the remainder of this section we establish conditions under which causal curves emanating from a spacelike submanifold will have timelike curves arbitrarily close to them.  To that end the following lemma is used repeatedly; cf. \cite[Lemma~45,~p.~293]{o1983}.  It is essentially a ``first derivative test" for determining the existence of timelike curves within a given variation.  Note that the curves dealt with in this lemma are smooth, not piecewise smooth.  Given a piecewise smooth variation $\xx\colon [0,b] \times (-\delta,\delta) \lra M$, $(u,v)\mapsto\xx(u,v)$, recall the notation that we established in Section \ref{section:curvature}: $\sigma_v=\xx(\cdot,v)$ and $\beta_u=\xx(u,\cdot)$ for $u\in[0,b]$ and $v\in(-\delta,\delta)$.

\begin{lemma}
\label{lemma:firstder}
Let $\sigma\colon [0,b] \lra M$ be a smooth future-pointing causal curve in a Finsler spacetime $(M,L)$ and $\xx\colon [0,b] \times (-\delta,\delta) \lra M$ a smooth variation of $\sigma$.  Assume that $L$ is smooth and has nondegenerate fundamental tensor on the velocities of the variational curves and let $V(u) =\dot\beta_u|_{v=0}$ be the variation field, $A(u) =D_{\beta_u}^{\dot\sigma_v}\dot\beta_u|_{v=0}$ the acceleration, and $f\colon [0,b] \times (-\delta,\delta) \lra \R$ the function given by $f(u,v)=L(\dot\sigma_v(u))$. Then
 \begin{align}
 \frac 12 \frac{\partial f}{\partial v}\bigg|_{v=0}&=\ g_{\ds}(V',\ds)\ =\ -g_{\ds}(V,D_{\sigma}^{\ds} \ds)\ +\ \frac{\partial}{\partial u} (g_{\ds}(V,\ds)),\label{firstvar}\\
 \frac 12 \frac{\partial^2 f}{\partial v^2}\bigg|_{v=0}&=\ g_{\ds}(A',\ds)\ -\ g_{\ds}(R^\sigma(\ds,V) V,\ds)\ +\ g_{\ds}(V',V').\label{secondvar}
 \end{align}
 Moreover, when $\sigma$ is a geodesic,
 \begin{equation}\label{secondvargeo}
 \frac 12 \frac{\partial^2 f}{\partial v^2}\bigg|_{v=0} =\ g_{\ds}(A',\ds)-g_{\ds}(V''-R^\sigma(\ds,V)\ds,V)\ +\ \frac{\partial}{\partial u}(g_{\ds}(V',V)).
 \end{equation}
 Finally, if $\ip{V'}{\ds}{\ds} > 0$ on $[0,b]$, then for sufficiently small $v \in (0,\delta)$ the smooth curves $\sigma_v\colon[0,b] \lra M$ will be future-pointing timelike.
\end{lemma}

\begin{proof}
We can assume without loss of generality that the variation is contained in the open subset $A^*$ that contains $\hat{A}$ and that $L$ can be extended as a nondegenerate Lorentz-Finsler metric in such a way that $L$ is negative in $A^*\setminus \hat{A}$ (recall Remark \ref{rem:Lnegative}).  Then the computation follows along the same lines as \cite[Lemma~45,~p.~293]{o1983}, taking into account \cite[Proposition 3.2]{Ja13} and \cite[Lemma 3.10]{JaSo14}.
\end{proof}

As an immediate consequence of Lemma~\ref{lemma:firstder}, we have the following proposition.

\begin{prop}
\label{prop:notnormal}
Let $P$ be a spacelike submanifold of a Finsler spacetime $(M,L)$ and $\sigma\colon [0,b] \lra M$ a future-pointing lightlike geodesic from $\sigma(0) \in P$ to $\sigma(b) = q$.  If $\sigma$ is not orthogonal to $P$ at $\sigma(0)$, then there is a smooth future-pointing timelike curve from $P$ to $q$.
\end{prop}

\begin{proof}
This proof goes through as it appears in \cite[Proposition~50,~p.~298]{o1983}, simply by considering the $\sigma$-parallel translate of some $y \in T_{\sigma(0)}P$ along $\sigma$.
\end{proof}

\begin{prop}
\label{prop:causal1}
Let $(M,L)$ be a Finsler spacetime and $\sigma\colon [0,b] \lra M$ a piecewise smooth future-pointing causal curve that is not a smooth lightlike pregeodesic.  Then there is a piecewise smooth future-pointing timelike curve from $\sigma(0)$ to $\sigma(b)$ arbitrarily close to $\sigma$.
\end{prop}

\begin{proof}
Let us assume first that $\sigma$ is piecewise smooth with $\ds(s_0) \in A_{\sigma(s_0)}$ for some $s_0\in [0,b]$.  Then we can assume without loss of generality that $s_0\in (0,b)$ and that $\sigma$ is smooth in some interval of $s_0$.  Let $d$ be the largest number such that $d<s_0$ and one of the one-side derivatives $\dot\sigma^\pm(d)$ is lightlike.  Choose an interval $[c',d'] \subset [0,s_0)$ which contains $d$ and  such that $\sigma$ belongs to the region where the fundamental tensor is well-defined and nondegenerate.  Let $W$ be a vector field along $\sigma|_{[c',d']}$ such that $g_{\ds}(W,\ds)>0$ and $f$ any smooth function on $[c',d']$ that vanishes at $c'$.
Define $V := fW$ and let $\xx\colon[c',d'] \times (-\delta,\delta) \lra M$ be any  variation of $\sigma$ which fixes the first point and has variation field $V$.  Now choose $f$ such that
$$
\ip{V'}{\ds}{\ds}\Big|_{[c',d']} =\ \dot{f}\ip{W}{\ds}{\ds}\ +\ f g_{\ds}(W',\ds)\Big|_{[c',d']} >\ 0, 
$$
which is always possible because $\ip{W}{\ds}{\ds}>0$ and we can apply the results of ODE theory. Extend $V$ to a vector field in $[c',s_0]$ in such a way that it is zero at $s_0$.  Then by Lemma~\ref{lemma:firstder}, for $v > 0$ sufficiently small the curves $\sigma_v$ are timelike \emph{on the interval $[c',d']$}. Observe that we can assume $\dot\sigma^+(d')$ is timelike by taking a bigger $d'$ if necessary.  Then on the remaining interval $[d',s_0]$ the curve $\sigma$ is itself timelike, so taking $v$ closer to 0 if necessary, the curves $\sigma_v$ will remain timelike on $[d',s_0]$ as well.  Proceeding by induction, note that it is always possible to obtain a curve from $\sigma(0)$ to $\sigma(s_0)$. Otherwise, there would exist a sequence $\{s_n\}\subset [0,s_0]$ which converges to $\bar{s}_0\in [0,s_0]$ such that all $\sigma(s_n)$ are chronologically connected to $\sigma(s_0)$ and such that no point on the curve $\sigma|_{[0,\bar{s}_0]}$ is chronologically connected to $\sigma(s_0)$.  In such a case, at least one of the derivatives $\dot\sigma({\bar s}_0)^\pm $ has to be lightlike.  By a similar process analyzing the possible cases,  we can obtain a  future-pointing timelike curve from $\sigma(\bar{s}_0)$ to $\sigma(s_0)$,  a contradiction.  To get a timelike curve from $\sigma(s_0)$ to $\sigma(b)$ we can proceed in an analogous way.   

 Assume now that $\sigma$ is a future-pointing smooth lightlike curve everywhere, but not a pregeodesic.  Then we proceed as follows.  Being future-pointing lightlike means that each $\ds(u) \in \hat{A}_{\sigma(u)}\setminus A_{\sigma(u)}$, hence
$$
\frac{d}{du}\underbrace{\ip{\ds}{\ds}{\ds}}_{0}\ =\ 2\ip{D_{\sigma}^{\ds}\ds}{\ds}{\ds}\ +\ \underbrace{2C_{\ds}(D_{\sigma}^{\ds}\ds,\ds,\ds)}_{0},
$$
so that each $D_{\sigma}^{\ds}\ds(u)$ is $g_{\ds(u)}$-orthogonal to $\ds(u)$.  Now consider the smooth function $u \mapsto \ip{D_{\sigma}^{\ds}\ds(u)}{D_{\sigma}^{\ds}\ds(u)}{\ds(u)}$.  This function can never be positive, because a $-g_{\ds(u)}$-timelike vector can never be $g_{\ds(u)}$-orthogonal to the $-g_{\ds(u)}$-lightlike vector $\ds(u)$.  Hence each $\ip{D_{\sigma}^{\ds}\ds(u)}{D_{\sigma}^{\ds}\ds(u)}{\ds(u)} \leq 0$.  If this function were identically zero, then $D_{\sigma}^{\ds}\ds$, like $\ds$, would be $-g_{\ds}$-lightlike at each $u \in [0,b]$, hence the two $g_{\ds}$-orthogonal vectors $\ds$ and $D_{\sigma}^{\ds}\ds$ would have to be collinear on $[0,b]$\,---\,but this would force $\sigma$ to be a lightlike pregeodesic in the Finsler spacetime $(M,L)$ (for example, adapt the proof in \cite[p.~95]{o1983} using \cite[Remark 4.3]{JaSo14}), contrary to our assumptions.  Hence the function $\ip{D_{\sigma}^{\ds}\ds(u)}{D_{\sigma}^{\ds}\ds(u)}{\ds(u)} \leq 0$ is not identically zero.  Consider an interval $[\bar a,\bar b]\subseteq [0,b]$ in which $\ip{D_{\sigma}^{\ds}\ds(u)}{D_{\sigma}^{\ds}\ds(u)}{\ds(u)} < 0$ and    pick a vector $W(\bar b) \in A_{\sigma(\bar b)}$  where $L$ is smooth with non-degenerate fundamental tensor,  which necessarily satisfies $\ip{W(\bar b)}{\ds(\bar b)}{\ds(\bar b)} > 0$ by part $(ii)$ of Proposition \ref{Lorentznorm}.  Let $W$ be its parallel translate along $\sigma$, so that $D_{\sigma}^WW = 0$  (as $W$ could enter the region where $L$ is not smooth or with degenerate fundamental tensor, we can take a smaller interval $[\bar a',\bar b]\subseteq [\bar a,\bar b]$ if necessary);  then $L(W)$ is constant along $\sigma|_{[\bar a,\bar b]}$, so that $W(u) \in A_{\sigma(u)}$ for all $u \in  [\bar a,\bar b]$ and hence $\ip{W}{\ds}{\ds} > 0$ on  $[\bar a,\bar b]$.  With a view to using Lemma~\ref{lemma:firstder}, we will now define smooth functions $f,h$ on $[\bar a,\bar b]$ such that the vector field $V := fW + hD_{\sigma}^{\ds}\ds$ satisfies $\ip{V'}{\ds}{\ds} > 0$ along $\sigma|_{[\bar a,\bar b]}$.  To that end, begin by noting that differentiation of $\ip{D_{\sigma}^{\ds}{\ds}}{\ds}{\ds} \equiv 0$ yields $\ip{D_{\sigma}^{\ds}D_{\sigma}^{\ds}\ds}{\ds}{\ds} = -\ip{D_{\sigma}^{\ds}\ds}{D_{\sigma}^{\ds}\ds}{\ds}$, which in turn implies that
$$
\ip{V'}{\ds}{\ds}\ =\ f'\ip{W}{\ds}{\ds}\ +\ f g_{\ds}(W',\ds)\ -\ h \ip{D_{\sigma}^{\ds}\ds}{D_{\sigma}^{\ds}\ds}{\ds},
$$
where we point out that $h'\ip{D_{\sigma}^{\ds}\ds}{\ds}{\ds} = 0$.  Define $u\mapsto q(u):=g_{\ds}(W',\ds)/g_{\ds}(W,\ds)$. Because  $\ip{D_{\sigma}^{\ds}\ds}{D_{\sigma}^{\ds}\ds}{\ds} < 0$ in $[\bar a,\bar b]$,  we can define $h$ to be any smooth function on $[\bar a,\bar b]$ vanishing at the endpoints and satisfying $\int_{\bar a}^{\bar b} h r e^{\int_{\bar a}^{u} q(t)\,dt} \,du = -\int_{\bar a}^{\bar b}e^{\int_{\bar a}^{u} q(t)\,dt}\,du.$  Next, define $f$ to be the smooth function 
\[u\ \mapsto\ f(u)\ :=\ \frac{\int_{\bar a}^u (hr+1)e^{\int_{\bar a}^{v} q(t)\,dt}\,dv}{e^{\int_{\bar a}^u q(t)\,dt}}\cdot\]
  Like $h$, $f$ vanishes at the endpoints, hence so does $V$.  Finally, a straightforward computation shows that
\beqa
\ip{V'}{\ds}{\ds}\ =\ \ip{W}{\ds}{\ds}\ >\ 0.\nonumber
\eeqa
Consider now any variation of $\sigma|_{[\bar a,\bar b]}$ with variation field $V$.  Then Lemma~\ref{lemma:firstder} applies to give a future-pointing timelike curve with endpoints the same as $\sigma|_{[\bar a,\bar b]}$.  When concatenated with $\sigma|_{[0,\bar a]}$ and $\sigma|_{[\bar b,b]}$, this gives a future-pointing causal curve from $\sigma(0)$ to $\sigma(b)$ that is timelike at some point, hence the first part of the proof applies. Finally, if $\sigma$ is a piecewise lightlike geodesic, the last part of the proof in \cite[Proposition~10.46,~p.~295]{o1983} adapts easily. 
\end{proof}

Observe that when $L$ is  defined on the whole tangent bundle $TM$ and smooth on $TM\setminus 0$ with nondegenerate  fundamental tensor, then convex neighborhoods are available \cite{whitehead32,whitehead33}, in which case Proposition \ref{prop:causal1} is easily obtained using  a generalization of \cite[Lemma 5.33]{o1983}.  This was first observed in \cite[Lemma 2]{Minguzopoli}.  An advantage of the proof given here in Proposition \ref{prop:causal1}, apart from its validity in the conic non-smooth case, is that  it allows one to find a timelike curve in the same causal homotopy as the original one. Observe that in some situations there can be infinite causal homotopy classes \cite{MoSa15}.  An immediate corollary to Proposition \ref{prop:causal1}, needed in Section~\ref{section:Penrose} below, is the following one (well known in the Lorentzian case, whose proof is identical).

\begin{cor}
\label{cor:IJ}
Let $(M,L)$ be a Finsler spacetime and $P \subset M$ any subset.  Then
$$
\chrf{P}\ =\ \chrf{\chrf{P}}\ =\ \chrf{\cf{P}}\ =\ \cf{\chrf{P}}\ \subset\ \cf{\cf{P}}\ =\ \cf{P}.
$$
Furthermore, ${\rm int}\,\cf{P} = \chrf{P}$.
\end{cor}

\begin{proof}
See \cite[p.~402]{o1983} and \cite[Lemma~6,~p.~404]{o1983}.
\end{proof}

Moving on, we now generalize Proposition~\ref{prop:causal1} to the case when the first endpoint is a spacelike submanifold. But first we need a technical lemma.

\begin{lemma}\label{AVvar}
If $\sigma\colon[0,b]\lra M$ is a geodesic of a Finsler spacetime $(M,L)$ and $V$ and $A$ are vector fields along $\sigma$ satisfying $V(a)=V(b)=0=A(a)=A(b)$, then there is a variation $\xx$ of $\sigma$ with variation field $V=\dot\beta_u|_{v=0}$ and acceleration $A=D_{\beta_u}^{\dot\sigma_v}\dot\beta_u|_{v=0}$.  Moreover, $\xx$ fixes the first and last points.
\end{lemma}
\begin{proof}
First observe that at every instant $s_0\in[a,b]$ there exists an open subset $U_{s_0}$ that admits a geodesic vector field tangent to $\sigma$ in $U_{s_0}$.  (For example, consider a hypersurface $H$ transverse to $\sigma$ and containing $\sigma(t_0)$, and then extend $\dot\sigma(t_0)$ to a transverse vector field $W$ in $H$ (shrinking $H$ if necessary). Then the geodesics departing from $H$ with velocities $W$ give the geodesic vector field in a neighborhood of $\sigma(t_0)$.)  Moreover, by the compactness of $[a,b]$ we can choose a partition $t_0=a<t_1<\cdots<t_k<t_{k+1}=b$ such that $\sigma|_{[t_i,t_{i+1}]}$ is contained in an open subset $U_i$ that admits a geodesic vector field $W_i$ tangent to $\sigma$ for $i=0,\dots,k$. Define $h_i=g_{W_i}$, which is a Lorentzian metric in $U_i$, and let ${\exp_i}$ be the exponential map associated to $h_i$. Observe that if $\alpha_i\colon(-\delta_i,\delta_i)\lra U_i$ is a curve such that $\alpha_i(0)=\sigma(t_i)$, $\dot{\alpha}_i(0)=V(t_i)$, and $D_{\alpha_i}^{W_i}\dot{\alpha}_i=A(t_i)$, then by making $\delta_i$ smaller if necessary we can assume that the image of $\alpha_i$ is embedded in $U_i$ and  contained in an open subset wherein $({\exp_i})_{\sigma(t_i)}$ is a diffeomorphism. Then if $\tilde{z}_i\colon [t_i,t_{i+1}]\times (-\delta_i,\delta_i)\lra T_{\sigma(t_i)}M$ is a function given by $\tilde{z}_i(u,v)= \frac{t_{i+1}-u}{t_{i+1}-t_i}(\exp_i)_{\sigma(t_i)}^{-1}(\alpha_i(v))$ and $z_i(u,v)=T_u(\tilde{z}_i(u,v))$ for every $(u,v)\in  [t_i,t_{i+1}]\times (-\delta_i,\delta_i)$, where $T_u:T_{\sigma(t_i)}M\rightarrow T_{\sigma(u)}M$ is any parallel transport along $\sigma$, define a variation $\xx_i\colon[t_i,t_{i+1}]\times (-\delta_i,\delta_i)\lra M$ by
\beqa
&&\xx_i(u,v)\ =\nonumber\\
&&(\exp_i)_{\sigma(u)}\left[\Big(z_i(u,v)+v\big(V(u)-\frac{\partial z_i}{\partial v}(u,0)\big)\Big)+\frac{1}{2}v^2\Big(A(u)-\frac{\partial^2 z_i}{\partial v^2}(u,0)\Big)\right].\nonumber
\eeqa
It is easy to see that the variation field of $\xx_i$ is $V$ and its acceleration $A$ (observe that the exponential map at 0 preserves velocities and accelerations up to canonical isomorphism, and if $\nabla^i$ is the Levi-Civita connection of $h_i$, then $\nabla^i_{W_i}=\nabla^{W_i}_{W_i}$ because $W_i$ is a geodesic vector field; see \cite[Lemma 7.4.1]{Sh01}).
Then we can begin by choosing $\alpha_0$ constant, thereby obtaining $\xx_0$ (which fixes the first point) and $\alpha_1(v)=\xx_0(t_1,v)$. Proceeding inductively with the choice $\alpha_i(v)=\xx_i(t_i,v)$, we get a sequence of variations that match continuously, thereby giving a variation that fixes the first and the last point.
\end{proof}
\begin{prop}
\label{prop:causal2}
Let $P$ be a spacelike submanifold of a Finsler spacetime $(M,L)$ and $\sigma\colon [0,b] \lra M$ a future-pointing lightlike geodesic from $\sigma(0) \in P$ to $\sigma(b) = q$ that is orthogonal to $P$.  If there is a focal point of $P$ along $\sigma$ strictly before $q$, then there is a piecewise smooth future-pointing timelike curve from $\sigma(0)$ to $q$ arbitrarily close to $\sigma$.
\end{prop}

\begin{proof}
Taking into account Lemma \ref{AVvar}, the proof follows closely that of \cite[Proposition~48,~p.~296]{o1983}.  The goal of the proof therein is to use the focal point to deform a piece of $\sigma$ so that it becomes future-pointing timelike.  However, in \cite[Proposition~48,~p.~296]{o1983} there are two minor gaps, which we now point out.  Let $r > 0$ be the first focal point of $P$ along $\sigma$, and $J$ a nonzero $P$-Jacobi field along $\sigma$ that satisfies $J(r) = 0$.  Assume for now that $J(0) \in T_{\sigma(0)}P$ and $J(0)\not=0$.  Then part $(i)$ of Lemma \ref{Jacobi-P2} ensures that $J|_{[0,r]}$ is $-g_{\ds}$-spacelike: $J(0) \in T_{\sigma(0)}P$ and $J(r) = 0$ are $-g_{\ds}$-spacelike, while on $(0,r)$, $J$ is $g_{\ds}$-orthogonal but never tangent to the $-g_{\ds}$-lightlike curve $\sigma$, hence must be $-g_{\ds}$-spacelike.  Next, since $J(r) = 0$, there exists a smooth vector field $Y$ along $\sigma$ such that
$$
J(u)\ =\ (r-u)Y(u)
$$
(see \cite[p.~33]{o1983}), where $Y(0) \neq 0$ since we are assuming that $J(0) \neq 0$  (note that \cite[Proposition~48,~p.~296]{o1983} handles only the case $J(0) = 0$).  The same is true of $Y(r)$, since 
$$
J'(r)\ =\ \big[\!-Y(u)\ +\ (r-u)Y'(u)\big]\Big|_{u=r}\ =\ -Y(r),
$$
which implies that $Y(r) \neq 0$, otherwise $J'(r) = 0 = J(r)$ and $J$ would be identically zero (see \cite[Lemma~3.14]{JaSo14}).  From here on out the proof follows \cite[Proposition~48,~p.~296]{o1983} word for word, the only exception being that the acceleration must be modified in order to have the second derivative of the energy of the variation positive in a compact neighborhood. This is done as follows.  Using the notation in the proof of \cite[Proposition~48,~p.~296]{o1983} adapted to our case, let $t^*=r+\delta$ and
\[\vep<-\ip{V''-R^{\sigma}(\ds,V)\ds}{V}{\ds}\,\]
in $[t^*/4,3t^*/4]$, and define the function 
\[
\rho(t)\ =\ \begin{cases}
\vep t,& 0\leq t\leq t^*/4, \\
\vep(-t+t^*/2), & t^*/4\leq t\leq 3t^*/4,\\
-\vep(t^*-t), & 3t^*/4\leq t\leq t^*.
\end{cases}
\]  Then a suitable choice for the acceleration is $A(u) := (\ip{V'(u)}{V(u)}{\ds(u)}-\rho(u))N(u)$ along $\sigma$. 
\end{proof}

 Propositions \ref{prop:notnormal}, \ref{prop:causal1}, and \ref{prop:causal2}, which are the mirrors of \cite[Lemma~50,~p.~298]{o1983}, \cite[Proposition~46,~p.~294]{o1983}, and \cite[Proposition~48,~p.~298]{o1983}, respectively, all converge to the following important theorem, which is identical to \cite[Theorem~51,~p.~298]{o1983} (see also \cite[Corollary~5,~p.~404]{o1983}), and plays a crucial role in Penrose's singularity theorem.

\begin{thm}
\label{thm:fhfocal}
Let $P$ be a spacelike submanifold of a Finsler spacetime $(M,L)$.  If $\sigma\colon[0,b] \lra M$ is a piecewise smooth future-pointing causal curve from $\sigma(0) \in P$ to $\sigma(b) = q \in \cf{P} \setminus \chrf{P}$, then $\sigma$ must be a future-pointing lightlike geodesic that is orthogonal to $P$ at $\sigma(0)$ and has no focal points of $P$ strictly before $q$. 
\end{thm}

\section{Penrose's singularity theorem in a Finsler spacetime}
\label{section:Penrose}
 The following fundamental results hold in the Finslerian setting exactly as they do in the Lorentzian setting, and therefore we state them without proof: (1) an achronal set $B$ is a closed topological hypersurface if and only if $B$ has no edge points (see \cite[Corollary~26,~p.~414]{o1983}); (2) the boundary of a future set, if nonempty, is a closed achronal topological hypersurface (see \cite[Corollary~27,~p.~415]{o1983}).  (Note that in proving (2) in the Lorentzian setting, one typically chooses local coordinates $(x^i)$ in which $\partial/\partial x^0$ is future-pointing timelike, this being guaranteed by virtue of a time orientation on the Lorentz manifold.  In the case of a Finsler spacetime, this is achieved by choosing a smooth timelike vector field $\tau$ for $(M,L)$, which is always possible by Proposition \ref{lemma:Lorentz}.) Next, we define Cauchy hypersurfaces as in the Lorentzian case.

\begin{defi}
\label{defi:cauchy}
Let $(M,L)$ be a Finsler spacetime.  A subset $B \subset M$ is a \emph{Cauchy hypersurface} if $B$ is met exactly once by every inextendible, piecewise smooth, timelike curve.
\end{defi}

 It follows exactly as in the Lorentzian case that any Cauchy hypersurface in a Finsler spacetime $(M,L)$ is a closed topological hypersurface (see \cite[Lemmas~29,30,~p.~415-6]{o1983}; note that Proposition~\ref{prop:causal1} is needed here).  Furthermore, by choosing a timelike vector field $\tau$ we can, just as in the Lorentzian setting, construct a continuous retraction $r\colon M \lra S$ by defining $r(p)$ to be the unique intersection point in $S$ of the (timelike) integral curve of $\tau$ through $p$ (the proof is identical to the Lorentzian case; see \cite[Proposition~31,~p.~417]{o1983}).  Before proceeding to Penrose's proof, we list a few further properties of the \emph{future horizons} $\fh{P} = \cf{P} \setminus \chrf{P}$.  As expected, they have direct Lorentzian analogues.

\begin{lemma}
\label{lemma:fh1}
Let $(M,L)$ be a Finsler spacetime and $P \subset M$ a nonempty achronal subset.  Then the following are true:
\begin{enumerate}
\item[{\rm (a)}] $\fh{P}$ is achronal and $P \subset \fh{P}$ (hence the latter is nonempty),
\item[{\rm (b)}] If $P$ is compact and the diamonds $\cf{p} \cap \cp{q}$ are closed for all $p,q \in M$, then $\fh{P}$, is a closed topological hypersurface.
\end{enumerate}
\end{lemma}

\begin{proof}
For (a), see \cite[p.~435]{o1983}; for (b), observe that $\cf{P}$ must be closed (cf. \cite[Proposition~4.3]{GG}) and apply \cite[Cor. 27, p.~415]{o1983} (note that the existence of limit curves is not required for this proof, or in Theorem~\ref{thm:Penrose} below).
\end{proof}

Fundamental to Penrose's proof is the concept of \emph{trapped surface}.  As in \cite[p.~435]{o1983}, we define trapped surfaces in terms of their mean curvature vector fields.  

\begin{defi}
\label{defi:ts}
Let $(M,L)$ be a Finsler spacetime and $P \subset M$ a spacelike codimension 2 submanifold.  Then $P$ is a \emph{trapped surface} if for every future-pointing lightlike vector $z$ orthogonal to $P$ (recall Proposition \ref{codim2}), the mean curvature vector field $H_z^P$ satisfies $g_z(H_z^P,z)>0$.
\end{defi}
\begin{rem}\label{rem:Chern}
\label{connections}
Observe that all the elements in Penrose's theorem, including trapped surfaces, are independent of the connection, provided that a certain family of affine  connections $\tilde{\nabla}^v$ (as in \cite[Def. 7.1.1]{Sh01}) is sufficiently compatible, in the sense that
\begin{enumerate}[(i)]
\item it gives the same geodesics (which do not depend on the connection as they are critical points of the energy functional),
\item it gives the same flag curvature,
\item and if we define the tensor $Q_v(X,Y)=\tilde{\nabla}^v_XY-\nabla^v_XY$,  where $\nabla^v$ is the Chern connection, for $v\in A\subset TM$ and $X,Y$ are vector fields in an open subset of $\pi(v)$, then $g_z(Q_z(x,y),z)=0$ for every lightlike vector $z$ and every $x,y\in T_{\pi(z)}M$.
\end{enumerate}
The last condition easily implies that $g_z(z,H^P_z)=g_z(z,\tilde{H}^P_z)$ for every lightlike vector $z$, where $\tilde{H}^P_z$ is the mean curvature vector of a non-degenerate submanifold $P$ computed with the connection $\tilde{\nabla}^v$. Moreover, as the scalar Ricci curvature is a mean of flag curvatures, the second condition implies also that the connections generate the same scalar Ricci curvature.   It is known that the classical connections used in Finsler geometry (Berwald, Cartan, Chern, and Hashiguchi) give the same value for the flag curvature  (\cite[Section 3.9]{bao2000}) and the same geodesics. Up to the third condition, it is easy to check for a Berwald connection, which can be easily interpreted as a family of affine connections (see \cite[Chapter 7]{Sh01}). Indeed, in this case the tensor $Q_v$ is the tensor $L_v$ defined in \cite[page 100]{Sh01}, since in this case, using \cite[Eq. (6.26)]{Sh01} and the symmetry of $L_v$,
\[g_v(L_v(x,y),v)=L_v(x,y,v)=0.\]
Moreover, using the relations between the connections given in \cite[page 39]{bao2000}, we can easily associate families of affine connections to Cartan and Hashiguchi, which are the same as for Chern and Berwald, respectively. It follows, then, that the third property also holds for Cartan and Hashiguchi connections.
\end{rem}
By Lemma~\ref{lemma:spacesub}, each $g_{H_z^P}|_{T_pP \times T_pP}$ is nondegenerate, thereby allowing the splitting
$$
T_pM\ =\ T_pP\ \oplus\ (T_pP)^{\perp}_{H_z^P}.
$$
The key step in Penrose's proof is that under ``reasonable" geometric conditions, a trapped surface $P$ will necessarily have a \emph{compact} future horizon $\fh{P}$.  The proof of this in a Finsler spacetime follows \cite[Proposition~60,~p.~436]{o1983}, but requires some modification.

\begin{prop}
\label{prop:ftrapped}
Let $(M,L)$ be a Finsler spacetime, $P \subset M$ a compact, achronal trapped surface, and suppose that
\begin{enumerate}
\item[{\rm(a)}] ${\rm Ric}(v) \geq 0$~\text{for all future-pointing lightlike vectors $v \in \hat{A}\setminus A$},
\item[{\rm(b)}] $M$ is future lightlike complete.
\end{enumerate}
Then the future horizon $\fh{P} = \cf{P} \setminus \chrf{P}$ is compact.
\end{prop}

\begin{proof}
Recall that the subset of orthogonal vectors to $P$, denoted by $TP^\perp$, is a submanifold of $TM$, (see \cite[Lemma 3.3]{JaSo14} and recall that by Lemma \ref{codim2}, $P_0=P$ in the present context), and that we have assumed for convenience that $L$ is defined in a conic open subset $A^*\supset \hat A$.  In fact, for the purposes of this proof, we can assume that $L$ is defined in a conic open subset which contains all lightlike vectors and that it is smooth therein.  Its intersection with the lightcone $\mathcal C = \cup_{p \in M} \,\mathcal{C}_p$ is a submanifold because $TP^{\perp}$ and $\mathcal{C}$ intersect transversely. This is because the tangent space to $\mathcal C$ at $v$ is the sum of the subspace of the vertical space of $TM$, 
\[T_{v}{\mathcal C}_{\pi(v)}\ =\ \{u\in T_{\pi(v)}M:  g_v(v,u)=0\},\]
 plus a  subspace $\tilde{U}$ transversal to the fibers of $TM$ of dimension $n$, while the tangent space to $TP^\perp$ contains the vertical vectors 
 \[T_{\pi(v)}P^{\perp g_v}\ :=\ \{u\in T_{\pi(v)}M: \text{$g_v(u,w)=0$ for every $w\in T_{\pi(v)}P$}\}\]
 (see the proof of \cite[Lemma 3.3]{JaSo14}). Now observe that  $T_{\pi(v)}P^{\perp g_v}$ cannot be contained in $T_{v}{\mathcal C}_{\pi(v)}$ because it has Lorentzian signature; hence $T_{\pi(v)}P^{\perp g_v}+T_{v}{\mathcal C}_{\pi(v)}=T_{\pi(v)}M$. It follows that $T_v{\mathcal C}+T_vTP^\perp\supseteq T_{\pi(v)}M+\tilde{U}=T_vTM$ and hence $TP^\perp$ and  $\mathcal C$ are transverse and their intersection a submanifold. Now, given $z\in TP^\perp\cap {\mathcal C}$, denote by $\sigma_{z}$ the future-pointing lightlike geodesic starting at $\sigma_{z}(0) =\pi(z) \in P$ with tangent $\ds_z(0) = z$.   Because $P$ is a trapped surface, we have that 
$$
k_{z}\ :=\ \ip{H^P_z}{z}{z}\ >\ 0
$$
(recall part $(ii)$ of Proposition  \ref{Lorentznorm}).  Hence by Proposition~\ref{prop:Ric} each $\sigma_{z}$ has a focal point somewhere along the interval $[0,1/k_{z}]$ (since $M$ is future lightlike complete, each $\sigma_{z}|_{[0,1/k_{z}]}$ is defined).  To find a common interval for all $\sigma_{z}$, let $g_R$ be any Riemannian metric on $M$, $SM$ the unit tangent bundle for this metric, and $U:=TP^\perp\cap {\mathcal C}\cap SM$, which is a submanifold of $TM$ because $SM$ and $TP^\perp\cap {\mathcal C}$ are transversal; furthermore, $U$ is compact because $P$ is compact and for any $p\in P$ there are only two vectors $v_1$, $v_2$ in $U$ that belong to $T_pM$ (see Proposition \ref{codim2}). Define a map $k\colon U \lra \RR$ by
$$
z\ \mapsto\ \ip{H^P_z}{z}{z}\ =\ k_{z}.
$$
This map is continuous (it's the restriction of the smooth map $k\colon TP^\perp \lra \RR$).   By compactness of $U$, there is a smallest value for $k|_{U}$, which we denote $1/b$.  We thus conclude that all the future-pointing lightlike geodesics $\sigma_{z}$ as defined above with $z\in U$ have a focal point somewhere in the interval $[0,b]$.
\vskip 12pt
 This now ensures that $E^+(P)$ must be compact, as follows.  Consider an arbitrary future-pointing lightlike geodesic $\gamma$ starting at $\gamma(0) = p \in P$ and orthogonal to $P$.  By Theorem~\ref{thm:fhfocal}, the set $E^+(P)$ is generated by future-pointing lightlike geodesics that are orthogonal to $P$ and having no focal points.  In the present context, this means that if $q \in E^+(P)$, then $q = \sigma_{z}(s_*)$, for some $\sigma_{z}$  starting at $P$ and $s_* \in [0,b]$.  Now define the set $K = \{ sz : z \in U ~\text{and}~0 \leq s \leq b\}$ and note that if $q \in \fh{P}$, then $q = \sigma_{z}(s_*) = \sigma_{s_*z}(1)$, so that $q \in \text{exp}(K)$, hence $\fh{P} \subset \text{exp}(K)$.  Moreover,  $\text{exp}(K)$ is compact because $K$ is compact (observe that $\text{exp}$ can be extended continuously by homogeneity to the zero section).  To show that $\fh{P}$ is compact, let $\{\tilde{q}_n\} \subset E^+(P)$ be any sequence.  Viewed as a sequence in the compact set $\text{exp}(K)$, it has a convergent subsequence $\{\tilde{q}_{n_j}\}$, with some limit point $\tilde{q} \in \text{exp}(K)$.  Then $\tilde{q} = \sigma_{sz}(1) = \sigma_{z}(s)$ for some $sz\in K$, which implies that $\tilde{q} \in \cf{P}$.  Now, if $\tilde{q} \in \chrf{P}$, then because $\chrf{P}$ is an open set we must have some $\tilde{q}_{n_j} \in \chrf{P}$, which cannot happen because $\{\tilde{q}_n\} \subset E^+(P)$.  Hence $\tilde{q} \in J^+(P) \setminus \chrf{P} = \fh{P}$, and the proof is complete.
\end{proof}

We are finally in a position to prove Penrose's singularity theorem for a Finsler spacetime.  As mentioned in the Introduction, we have not attempted in this paper to demonstrate the equivalence, in a Finsler spacetime, of global hyperbolicity and the existence of a Cauchy hypersurface.  In Penrose's proof, the key consequence of global hyperbolicity is that the sets $J^{\pm}(K)$ are closed whenever $K \subset M$ is compact; this follows if one postulates that the diamonds $\cf{p} \cap \cp{q}$ are closed for all $p,q \in M$, which is of course part of the definition of global hyperbolicity (for a complete treatment of the causal hierarchy of Lorentzian spacetimes, see \cite{ming08}; see also \cite{fs12} for an analysis via the cone structure).  In our proof below we have therefore assumed that our Finsler spacetimes are globally hyperbolic, in \emph{addition} to having Cauchy hypersurfaces (in fact it would have sufficed to assume, in place of global hyperbolicity, the weaker condition that the diamonds $\cf{p} \cap \cp{q}$ be closed for all $p,q \in M$).  Let us also observe that the Ricci curvature condition ((c) in Theorem \ref{thm:Penrose} below) follows from a condition of non-negativity along lightlike directions of the stress-energy tensor  appearing in the Finsler gravity field equations in \cite[Eq. (54)]{pfeifer}.  Whether or not the non-negativity of the stress-energy tensor can be interpreted as a null energy condition in the Finsler realm is something deserving of further study in future work.

\begin{thm}
\label{thm:Penrose}
Let $(M,L)$ be a globally hyperbolic Finsler spacetime and suppose that the following hold:
\begin{enumerate}
\item[{\rm (a)}] $(M,L)$ contains a noncompact Cauchy hypersurface $S$,
\item[{\rm (b)}] $(M,L)$ contains a compact, achronal trapped surface $P$,
\item[{\rm (c)}] $\text{{\rm Ric}}(v) \geq 0$ for all future-pointing lightlike vectors $v \in \hat{A} \setminus A$.
\end{enumerate}
Then $(M,L)$ is future lightlike incomplete.
\end{thm}
\begin{proof}
Since by assumption $(M,L)$ is globally hyperbolic, it follows easily that all $J^{\pm}(p)$'s are closed.
Then taking into account Lemma~\ref{lemma:fh1}, Proposition~\ref{prop:ftrapped}, and Proposition~\ref{lemma:Lorentz}, the proof goes through as in \cite[Theorem~61,~p.~436]{o1983}.
\end{proof}

\section{Concluding Remarks}
In this paper we have shown that Penrose's singularity theorem \cite{pen65} also holds on Finsler spacetimes.  Our definition of Finsler spacetime (Definition \ref{def:Finsler}) is general, being defined only along causal directions, and with a high degree of non-smoothness allowed along timelike directions; the latter fact in particular implies that our result holds for static Finsler spacetimes.  From this definition, we systematically established all the relevant causal concepts and properties required in Penrose's theorem, from the designations of spacelike, timelike, and lightlike vectors and submanifolds to Cauchy hypersurfaces and trapped surfaces.  Though much of this was analogous to the Lorentzian setting, nevertheless there are subtleties that are unique to the Finslerian setting, which we have been careful to point out (one such example is the notion of lightlike directions orthogonal to a spacelike surface, Proposition \ref{codim2}, which is a nontrivial fact to establish on a Finsler spacetime).  Next, we established the relevant variational results, from Jacobi fields to focal points, culminating in Theorem \ref{thm:fhfocal}, which parallels its Lorentzian version in \cite[Theorem~51,~p.~298]{o1983}.  Here, too, however, there are difficulties with respect to the curvature and exponential map that are unique to the Finslerian setting (see, e.g., Proposition \ref{prop:causal1} and Lemma \ref{AVvar}).  With these difficulties out of the way, we then proceeded to prove Penrose's singularity theorem on our Finsler spacetime analogously to the variational proof to be found in \cite[Theorem~61,~p.~436]{o1983}.  In doing so, we inserted the condition of global hyperbolicity into the assumptions of our theorem, since in this paper we have not established the equivalence, in a Finsler spacetime, of global hyperbolicity and the existence of a Cauchy hypersurface.  Finally, though we have used the Chern connection, our definitions of geodesics, Ricci curvature, and trapped surfaces do not depend on this particular choice of connection (see Remark \ref{rem:Chern}).

\section*{Acknowledgements}
The authors thank Miguel S\'anchez for helpful discussions, as well as the anonymous referees for valuable comments and suggestions.

\quad\vspace{0.5cm}\,\newline
\includegraphics[height=0.07\textheight]{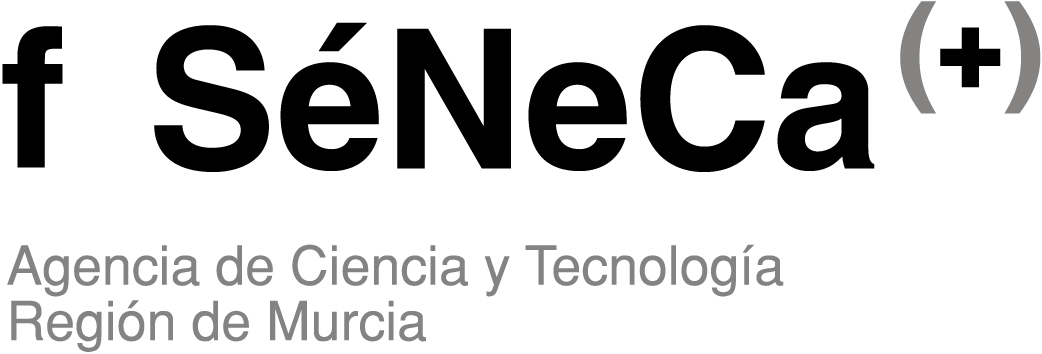}
\end{document}